\documentclass[reqno]{amsart}
\usepackage{amssymb}
\usepackage{hyperref}
\usepackage[all]{xy}
\newtheorem{theorem}{Theorem}[section]
\newtheorem{lemma}[theorem]{Lemma}
\newtheorem{proposition}[theorem]{Proposition}
\newtheorem{corollary}[theorem]{Corollary}
\theoremstyle{definition}
\newtheorem{definition}[theorem]{Definition}
\newtheorem{example}[theorem]{Example}

\theoremstyle{remark}

\numberwithin{equation}{section}

\begin{document}
%%%%%%%%%%%%%%%%%%%%%%%%%%

\title{ Supercyclicity of the left and right multiplication operators on Banach ideal of operators}
\author{ Mohamed Amouch and Hamza Lakrimi}
\address{Mohamed Amouch and Hamza Lakrimi,
University Chouaib Doukkali.
Department of Mathematics, Faculty of science
Eljadida, Morocco}
\email{amouch.m@ucd.ac.ma}
\email{hamza.lakrimi.hafdi@gmail.com}

\begin{abstract} Let $X$ be  a Banach space with $\dim X>1$ such that $X^{\ast}$, its dual, is separable and  $\mathcal{B}(X)$ the algebra of bounded linear operators on $X$. In this paper, we study the passage of property of being supercyclic from an operator $T\in\mathcal{B}(X)$  to the left and right  multiplication induced by $T$ on separable admissible Banach ideal of $\mathcal{B}(X)$. We give  a sufficient condition for the tensor product $T\widehat{\otimes}R$ of two operators to be  supercyclic. As a consequence,  we give another  equivalent conditions for the Supercyclicity Criterion.
\end{abstract}

%-----------------------------------

\maketitle

\section{Introduction and Preliminary}
Throughout the paper, we denote by $X$ a Banach space  with $\dim X>1$ such that $X^{\ast}$, its dual, is separable, $\mathcal{B}(X)$  denote the algebra of bounded linear operators on $X$ and $\mathcal{K}(X)$ denote the algebra of compact
operators on $X$. From \cite[Proposition 2.8]{Fabian}, if $X^{\ast}$ is separable, then $X$ is separable. Let  $T\in\mathcal{B}(X)$, the orbit of a vector $x\in X$ under $T$ is the set $$Orb(T, x):=\{T^{n}x : \ n\in\mathbb{N}\}.$$
An operator  $T\in\mathcal{B}(X)$ is said to be hypercyclic if there is some vector $x\in X$ such that $Orb(T, x)$ is dense in $X$; such a vector $x$ is called a hypercyclic vector for $T$. Similarly, $T\in\mathcal{B}(X)$ is said to be supercyclic if there is some vector $x\in X$ such  that $$\mathbb{C}Orb(T, x):=\{\alpha T^{n}x : \ \alpha\in\mathbb{C}, n\in\mathbb{N}\}$$  is dense in $X$; such a vector  $x$ is called supercyclic vector for $T$. From  \cite{H. M. Hilden}, an operator $T$ on $X$ is supercyclic if and only if for each pair $(U, V )$ of nonempty open subsets of $X$ there exist $\alpha\in\mathbb{C}$ and $n\in\mathbb{N}$ such that
 $$\alpha T^{n}(U)\cap V\neq\emptyset.$$
There is a characterization of hypercyclicity called  hypercyclicity criterion. It provides several sufficient conditions that ensure hypercyclicity. This criterion has been established by Carol Kitai in 1982 \cite{C. Kitai}. It is improved by several authors afterwards as Gethner and Shapiro  \cite{Gethner RM Shapiro JH} in 1987  and Juan Bes \cite{Bès JP} in 1998.  Recall that  $T\in \mathcal{B}(X)$ satisfy the hypercyclicity criterion, if there exist two dense subsets $Y$ and $Z$ in $X$, a strictly increasing sequence $(n_{k})_{k\geq0}$ of positive integers and mappings $S_{n_{k}}$ : $Z\longrightarrow X$ such that :
\begin{enumerate}
  \item[$(i)$]$T^{n_{k}}x\longrightarrow0$ for any $x\in Y$;
  \item[$(ii)$]$(S_{n_{k}}y)_{k}$ converges to 0 for any $y\in Z$;
  \item[$(iii)$]$(T^{n_{k}}S_{n_{k}}y)_{k}$ converges to $y$ for any $y\in Z$.
\end{enumerate}
 In the same way, Salas in 1999 gave a characterization of supercyclic bilateral backward weighted shifts via the Supercyclicity Criterion, that is a sufficient condition for supercyclicity \cite{salas}.  We say  $T\in \mathcal{B}(X)$ satisfy the supercyclicity criterion, if there exist two dense subsets $D_1$ and $D_2$ in $X$, a sequence $(n_{k})_{k\geq0}$  of positive integers, and also there exist a mappings $S_{n_{k}}$ : $D_2\longrightarrow X$ such that :
\begin{enumerate}
  \item[$(i)$]$\|T^{n_{k}}x\|\|S_{n_{k}}y\|\longrightarrow 0$ for every $x\in D_1$ and $y\in D_2$;
  \item[$(ii)$]$T^{n_{k}}S_{n_{k}}y\longrightarrow y$ for every $y\in D_2$.
  \end{enumerate}
For a more general overview of hypercyclicity, supercyclicity and related properties in linear dynamics, we refer to \cite{Bayart F, T. Bermudez, N. S. Feldman, Grosse-Erdmann KG, H. M. Hilden, Montes-Rodriguez, Rolewics S}.\\
The left multiplication operator $L_{T} : \mathcal{B}(X)\longrightarrow\mathcal{B}(X) $ induced by a fixed
bounded linear operator $T\in\mathcal{B}(X)$ is defined as $L_{T}(A)=TA$ and the right multiplication operator $R_{T} : \mathcal{B}(X)\longrightarrow\mathcal{B}(X) $ induced by a fixed
bounded linear operator $T\in\mathcal{B}(X)$ is defined as $R_{T}(A)=AT$ where $A\in\mathcal{B}(X)$. Recall \cite{Gohberg} that $(J, \|.\|_{J} )$ is said to be a Banach ideal of $\mathcal{B}(X)$ if :
 \begin{enumerate}
   \item[$(i)$] $J\subset\mathcal{B}(X)$ is a linear subspace;
   \item[$(ii)$] The norm $\|.\|_{J}$ is complete in $J$ and $\|S\|\leq\|S\|_{J}$ for all $S\in J $;
   \item[$(iii)$] $\forall S\in J$, $\forall A, B\in\mathcal{B}(X)$, $ASB\in J$ and  $\|ASB\|_{J}\leq\|A\|\|S\|_{J}\|B\|$;
   \item[$(iv)$] The rank one operators $x\otimes x^{\ast}\in J$ and $\|x\otimes x^{\ast}\|_{J} = \|x\|\|x^{\ast}\|$ for all  $x\in X$ and $x^{\ast}\in X^{\ast}$.
 \end{enumerate}
A rank one operator defined on $X$ as $(x\otimes x^{\ast})(z)=\langle x^{\ast}, z\rangle x=x^{\ast}(z)x$ for all $x\in X$, $x^{\ast}\in X^{\ast}$ and any $z\in X$. The space of finite rank operators $\mathcal{F}(X)$ is defined as the linear span of the rank one operators, that is
 $\mathcal{F}(X)=\{\sum_{i=1}^{m}x_{i}\otimes x_{i}^{\ast}, \ x_{i}\in X, \ x_{i}^{\ast}\in X^{\ast}, m\geq 1\}$. A Banach ideal $J$ of operators is said to be admissible if $\mathcal{F}(X)$ is dense in $J$ with respect to the norm $\|.\|_{J}$. Note that   $L_{T} : J\longrightarrow J$ and  $R_{T} : J\longrightarrow J$  are well-defined and we have  for all $A\in J$ :
 $$\|L_{T}(A)\|_{J}=\|TA\|_{J}\leq\|T\|\|A\|_{J} \ \ ;  \ \ \|R_{T}(A)\|_{J}=\|AT\|_{J}\leq\|T\|\|A\|_{J}.$$
 The study of hypercyclicity in operator algebras was initiated by Kit Chan in 1999 \cite{CH}, who showed that hypercyclicity can occur on the operator algebra $\mathcal{B}(H)$ with strong operator topology, when  $H$ is a separable Hilbert space. Subsequently  his idea was used by several authors, see for examples \cite{Jose bonet,Chan2001, Gupta2016, Petterson, Youseffi}. Recently, Gilmore et al in \cite{Gilmore2019,Gilmore2017}, investigate and study the hypercyclicity properties of the commutator maps $L_{T}-R_{T}$ and the generalised Derivations   $L_{A}-R_{B}$ built from the basic multiplications acting on separable Banach ideals of operators.\\
  The motivation of this study is that  Bonet  et al.  \cite{Jose bonet} use tensor product techniques developed in
  \cite{A. Peris} to  characterized the  hypercyclicity of the left and the right multiplication operators on  an admissible Banach ideal of operators.\\
  Let $E$ and $F$ be normed linear spaces. Recall \cite{R. A. Ryan} that  the projective tensor norm on $E\bigotimes F$ is the function $\Pi : E\bigotimes F\longrightarrow [0, +\infty[$ defined for all $z\in E\bigotimes F$ $$\Pi(z):=\inf\{\sum_{j=1}^{n}\|x_{j}\|\|y_{j}\| \ : \ z=\sum_{j=1}^{n}x_{j}\otimes y_{j} \}.$$
  For elementary tensors $z=x\otimes y$ we just have $\Pi(z)=\|x\|\|y\|$, with this topology the space is denoted by $E\bigotimes_{\pi}F$, and its completion by $E\widehat{\bigotimes_{\pi}}F$. For a more general overview of the projective tensor norm  and  its related properties we refer to \cite{H. Jarchow, R. A. Ryan}.\\
  The purpose of this paper is to characterize the supercyclicity of the left and the right multiplication on a separable admissible  Banach ideal of operators and give a sufficient condition for the tensor product $T\widehat{\otimes}R$ of two operators to be  supercyclic and some equivalent conditions for the Supercyclicity Criterion.\\
 In section 2, we  study the passage of property of being supercyclic from $T\in\mathcal{B}(X)$  to the left and the right  multiplication induced by $T$ on a separable admissible Banach ideal $J$ of $\mathcal{B}(X)$. So, we prove that :
\begin{enumerate}
  \item[$(i)$] $T$ satisfies the supercyclicity criterion on $X$ if and only if  $L_{T}$ is supercyclic on $(J, \|.\|_{J})$.
  \item[$(ii)$] $T^{\ast}$ satisfies the supercyclicity criterion on $X^{\ast}$ if and only if $R_{T}$ is supercyclic on $(J, \|.\|_{J})$.
  \end{enumerate}
  In section 3, we give  a sufficient condition for the tensor product $T\widehat{\otimes}R$ of two operators to be  supercyclic. As a consequence,  we give some equivalent conditions for the Supercyclicity Criterion.
\section{Supercyclicity of the left and right multiplication on Banach ideal of operators.}
We begin this section  with the following lemma which will be used in sequel.
\begin{lemma}\label{lemma1}
Let  $X$ be  Banach space, $X^{\ast}$, its dual, is separable, $(J, \|.\|_{J})$ an admissible Banach ideal of $\mathcal{B}(X)$. If  $D$ and $\Phi$ are a countable dense subsets of $X$ and $X^{\ast}$, respectively. Then the set
$$\mathcal{X}:=span\{x\otimes \phi /  x\in D ,  \phi\in \Phi\}$$
is a countable dense subset of $J$ with respect to $\|.\|_{J}$-topology.
\end{lemma}
\begin{proof}
Let $T\in J$. If $\epsilon>0$ is  arbitrary, then there is a finite rank operator $F$ such that $\|T-F\|_{J}<\frac{\epsilon}{2}$. Let $F=\displaystyle\sum_{i=1}^{N}\alpha_{i}a_{i}\otimes \varphi_{i}$, where $a_{i}\in X$,  $\varphi_{i}\in X^{\ast}$ and $\alpha_{i}\in \mathbb{C}$ for $i=1, 2, ..., N$. For every $i\in\{1, 2, ..., N\}$, there exist some $\phi_{i}\in\Phi$ such that $\|\varphi_{i}-\phi_{i}\|<\frac{\epsilon}{4N|\alpha_{i}|\|a_{i}\|}$ and  there exist $x_{i}\in D$ such that $\|a_{i}-x_{i}\|<\frac{\epsilon}{4N|\alpha_{i}|\|\phi_{i}\|}$. Therefore
\begin{eqnarray*}
\|F-\sum_{i=1}^{N}\alpha_{i}x_{i}\otimes \phi_{i}\|_{J}&=& \|\sum_{i=1}^{N}\alpha_{i}a_{i}\otimes \varphi_{i}-\sum_{i=1}^{N}\alpha_{i}x_{i}\otimes \phi_{i}\|_{J}\\
&=& \|\sum_{i=1}^{N}\alpha_{i}(a_{i}\otimes \varphi_{i}-x_{i}\otimes \phi_{i})\|_{J} \\
   &\leq&  \sum_{i=1}^{N}|\alpha_{i}| \|a_{i}\otimes \varphi_{i}-x_{i}\otimes \phi_{i}\|_{J}\\
   &=& \sum_{i=1}^{N}|\alpha_{i}| \|a_{i}\otimes \varphi_{i}-a_{i}\otimes \phi_{i}+a_{i}\otimes \phi_{i}-x_{i}\otimes \phi_{i}\|_{J}\\
   &=& \sum_{i=1}^{N}|\alpha_{i}| \|a_{i}\otimes( \varphi_{i}- \phi_{i})+(a_{i}-x_{i})\otimes \phi_{i}\|_{J} \\
   &\leq& \sum_{i=1}^{N}|\alpha_{i}| (\|a_{i}\otimes( \varphi_{i}- \phi_{i})\|_{J}+\|(a_{i}-x_{i})\otimes \phi_{i}\|_{J}) \\
  &=& \sum_{i=1}^{N}|\alpha_{i}| (\|a_{i}\| \| \varphi_{i}- \phi_{i}\|+\|a_{i}-x_{i}\|\| \phi_{i}\|) \\
   &<&\frac{\epsilon}{4}+\frac{\epsilon}{4}\\
   &=&\frac{\epsilon}{2}.
\end{eqnarray*}
Hence \begin{eqnarray*}
        \|T-\sum_{i=1}^{N}\alpha_{i}x_{i}\otimes \phi_{i}\|_{J} &=& \|T-F+F-\sum_{i=1}^{N}\alpha_{i}x_{i}\otimes \phi_{i}\|_{J} \\
         &\leq& \|T-F\|_{J}+\|F-\sum_{i=1}^{N}\alpha_{i}x_{i}\otimes \phi_{i}\|_{J} \\
         &<&  \frac{\epsilon}{2}+\frac{\epsilon}{2}\\
         &=& \epsilon.
      \end{eqnarray*}
Thus $\mathcal{X}$ is a countable dense subset of $J$ with respect to $\|.\|_{J}$-topology.
\end{proof}
In the setting of Banach ideals, J. Bonet et al \cite{Jose bonet} characterised the hypercyclicity of the left and the right multipliers using tensor techniques developed in \cite{A. Peris}. For a separable admissible Banach ideal $J$ of $\mathcal{B}(X)$, they showed that
\begin{enumerate}
  \item[$(i)$] $T$ satisfies the hypercyclicity criterion on $X$ if and only if  $L_{T}$ is hypercyclic on $(J, \|.\|_{J})$.
  \item[$(ii)$] $T^{\ast}$ satisfies the hypercyclicity criterion on $X^{\ast}$ if and only if $R_{T}$ is hypercyclic on $(J, \|.\|_{J})$.
  \end{enumerate}
In the following, we prove that this results hold for supercyclicity.
\begin{theorem}\label{theorem1}
Let $X$ be  Banach space, with $\dim X>1$ such that $X^{\ast}$, its dual, is separable and $T\in \mathcal{B}(X)$. Then  $T$ satisfy the supercyclicity criterion on $X$ if and only $L_{T}$ is supercyclic on $(J, \|.\|_{J})$.
\end{theorem}
\begin{proof}
$\Rightarrow)$ Assume that $T$ satisfy the supercyclicity criterion on $X$, then there exist a strictly increasing sequence $(n_{k})_{k}$ of positive integers, dense subsets $D_{1}, D_{2}$ of $X$ and maps $S_{n_{k}} : D_2\longrightarrow X$ such that for  all $x\in D_1$ and $y\in D_2$
      \begin{enumerate}
  \item[a)] $\|T^{n_{k}}x\|\|S_{n_{k}}y\|\longrightarrow 0$, as $k\rightarrow +\infty$;
  \item[b)] $T^{n_{k}}S_{n_{k}}y\longrightarrow y$, as $k\rightarrow +\infty$.
\end{enumerate}
Let $\Phi$ be dense subset of $X^{\ast}$ and consider the sets $X_{0}=span\{x\otimes \varphi / x\in D_{1}, \varphi\in\Phi\}$ and $Y_{0}=span\{y\otimes \phi / y\in D_{2}, \phi\in\Phi\}$ and the maps $Q_{n_{k}} : Y_{0}\longrightarrow J$ define by $$Q_{n_{k}}(\sum_{j=1}^{N}\beta_{j}y_{j}\otimes \phi_{j})=\sum_{j=1}^{N}\beta_{j}(S_{n_{k}}y_{j}\otimes \phi_{j}).$$
By  Lemma \ref{lemma1},  $X_{0}$ and $Y_{0}$ are subsets of $J$ which are $\|.\|_{J}$-dense  in $J$. Let $\displaystyle A=\sum_{i=1}^{N_{1}}\alpha_{j}x_{i}\otimes \varphi_{i}\in X_{0}$ and $\displaystyle B=\sum_{j=1}^{N_{2}}\beta_{j}y_{j}\otimes \phi_{j}\in Y_{0}$, then
\begin{eqnarray*}
  \|(L_{T})^{n_k}A\|_{J}\|Q_{n_{k}}B\|_{J} &=& \|(L_{T})^{n_k}(\sum_{i=1}^{N_{1}}\alpha_{i}x_{i}\otimes \varphi_{i})\|_{J}\|Q_{n_{k}}(\sum_{j=1}^{N_2}\beta_{j}y_{j}\otimes \phi_{j})\|_{J} \\
  &=& \|\sum_{i=1}^{N_{1}}\alpha_{i}T^{n_k}x_{i}\otimes \varphi_{i}\|_{J}\|\sum_{j=1}^{N_2}\beta_{j}(S_{n_{k}}y_{j}\otimes \phi_{j})\|_{J} \\
   &\leq& (\sum_{i=1}^{N_{1}}|\alpha_{i}|  \|T^{n_k}x_{i}\otimes \varphi_{i}\|_{J})(\sum_{j=1}^{N_2}|\beta_{j}| \|S_{n_{k}}y_{j}\otimes \phi_{j}\|_{J}) \\
  &=&  (\sum_{i=1}^{N_{1}}|\alpha_{i}|  \|T^{n_k}x_{i}\| \|\varphi_{i}\|)(\sum_{j=1}^{N_2}|\beta_{j}| \|S_{n_{k}}y_{j}\| \|\phi_{j}\|)\\
    &=&\sum_{i\leq N_{1}; j\leq N_{2}}|\alpha_{i}| |\beta_{j}| \|T^{n_k}x_{i}\|  \|S_{n_{k}}y_{j}\| \|\varphi_{i}\|\|\phi_{j}\|.
\end{eqnarray*}
Using the assumption $a)$. We show that $\|(L_{T})^{n_k}A\|_{J}\|Q_{n_{k}}B\|_{J}\longrightarrow0$,  as $k\longrightarrow +\infty$. In the other hand we have :
\begin{eqnarray*}
  L_{T}^{n_k}Q_{n_{k}}(B) &=& L_{T}^{n_k}Q_{n_{k}}(\sum_{j=1}^{N_{2}}\beta_{j}y_{j}\otimes \phi_{j}) \\
   &=&  L_{T}^{n_k}(\sum_{j=1}^{N_{2}}\beta_{j}S_{n_k}y_{j}\otimes \phi_{j})\\
  &=&\sum_{j=1}^{N_{2}}\beta_{j}T^{n_k}S_{n_k}y_{j}\otimes \phi_{j} \\
   &\longrightarrow& \sum_{j=1}^{N_{2}}\beta_{j}y_{j}\otimes \phi_{j}=B  \ \ as \ k\longrightarrow +\infty,
\end{eqnarray*}
by using the assumption $b)$. Hence $L_{T}$ satisfies the supercyclicity criterion on $(J, \|.\|_{J})$. Thus $L_{T}$ is supercyclic on $(J, \|.\|_{J})$.\\
$\Leftarrow)$ Suppose that $L_{T}$ is supercyclic  on $(J, \|.\|_{J})$. Assume that $x_1, x_2\in X$ are linearly independent  and define $$\begin{array}{ccccc}
\varphi & : &  J & \longrightarrow & X\bigoplus X\\
 & & R & \longmapsto & Rx_1 \oplus Rx_2.  \\
\end{array}$$
Then $\varphi$ is surjective. Indeed, let $y_1, y_2\in X$. By using the Hahn-Banach theorem, there exist  $x_{1}^{\ast}, x_{2}^{\ast}\in X^{\ast}$ such that $x_{1}^{\ast}(x_{1})=x_{2}^{\ast}(x_{2})=1$ and , $x_{1}^{\ast}(x_{2})=x_{2}^{\ast}(x_{1})=0$. Let  $R=y_{1}\otimes x_{1}^{\ast}+y_{2}\otimes x_{2}^{\ast}\in J$ thus  $\varphi(R)=Rx_1 \oplus Rx_2=y_{1}\oplus y_{2}$. For  $A\in J$, we have \begin{eqnarray*}
                                                                    (\varphi\circ L_{T})A &=&  \varphi(TA)  \\
                                                                    &=& (TA)x_1\oplus (TA)x_2 \\
                                                                     &=& T(Ax_1)\oplus T(Ax_2) \\
                                                                     &=& (T\oplus T)(Ax_1\oplus Ax_2)) \\
                                                                     &=&  (T\oplus T)\circ\varphi(A)\\
                                                                     &=&  ((T\oplus T)\circ\varphi)A.
                                                                  \end{eqnarray*}

 Therefore, $\varphi\circ L_{T}=(T\oplus T)\circ\varphi$. Thus $T\oplus T$ is supercyclic on $X\bigoplus X$. Hence, by \cite[Lemma 3.1]{T. Bermudez} $T$ satisfies the supercyclicity criterion.
\end{proof}
We have the following corollary.
\begin{corollary}
Let $X$ be  Banach space, with $\dim X>1$ such that $X^{\ast}$, its dual, is separable  and $T\in \mathcal{B}(X)$. Then the following are equivalent :
 \begin{enumerate}
   \item[$(i)$]$T$ satisfies the supercyclicity criterion on $X$.
   \item[$(ii)$] $L_{T}$ is supercyclic on $(\mathcal{K}(X), \|.\|)$ endowed with the norm operator topology.
   \item[$(iii)$] $L_{T}$ is supercyclic on $\mathcal{B}(X)$ in the strong operator topology.
 \end{enumerate}
\end{corollary}
\begin{proof}
$(i)\Leftrightarrow(ii)$ Consequence of  Theorem \ref{theorem1}. Since $\mathcal{K}(X)$ is an admissible Banach ideal of $\mathcal{B}(X)$.\\
 $(i)\Rightarrow(iii)$ Suppose that $T$ satisfy the supercyclicity criterion on $X$. Let $U$ and $V$ be two non-empty open subsets of $\mathcal{B}(X)$ in the strong operator topology. Since $\mathcal{K}(X)$ is dense in $\mathcal{B}(X)$ with the strong operator topology \cite[Corollary 3]{Chan2001},  there exist $A_{1}, A_{2}\in\mathcal{K}(X)$ such that $A_{1}\in U$ and $A_{2}\in V$. Thus we can find $x_1, x_2\in X\backslash\{0\}$ and $\epsilon_1, \epsilon_2>0$ such that $$\{A\in\mathcal{B}(X) \ : \ ||(A-A_1)x_1||<\epsilon_1\}\subset U$$
and $$\{A\in\mathcal{B}(X) \ : \ ||(A-A_2)x_2||<\epsilon_2\}\subset U.$$
Let $$U_i =\{A\in\mathcal{K}(X) \ : \ ||A-A_i||<\frac{\epsilon_i}{||x_i||}\}.$$
$U_i$ is a non-empty open subset of $\mathcal{K}(X)$ with the norm operator topology. By  Theorem \ref{theorem1} with $J=\mathcal{K}(X)$, $L_{T}$ is supercyclic on $\mathcal{K}(X)$, so there is some $\alpha\in\mathbb{C}$ and $n\in \mathbb{N}$ such that $$\alpha(L_{T})^{n}U_{1}\cap U_{2}\neq\emptyset.$$
Hence, it follows that $\alpha(L_{T})^{n}U\cap V\neq\emptyset$. Thus, $L_{T}$ is supercyclic on $\mathcal{B}(X)$ in the strong operator topology.\\
$(i)\Rightarrow(iii)$ By the same technique  as in the proof of  Theorem \ref{theorem1}.
\end{proof}
\begin{theorem}\label{theorem2}
Let $X$ be  Banach space, with $\dim X>1$ such that $X^{\ast}$, its dual, is separable and $T\in \mathcal{B}(X)$. Then $T^{\ast}$ satisfies the supercyclicity criterion on $X^{\ast}$ if and only $R_{T}$ is supercyclic on $(J, \|.\|_{J})$.
\end{theorem}
\begin{proof}
$\Rightarrow)$ Assume that $T^{\ast}$ satisfy the supercyclicity criterion on $X^{\ast}$, then there exist a strictly increasing sequence $(n_{k})_{k}$ of positive integers, dense subsets $\Phi_{1}, \Phi_{2}$ of $X^{\ast}$ and maps $M_{n_{k}} : \Phi_2\longrightarrow X^{\ast}$ such that for  all $\varphi\in \Phi_1$ and $\phi\in \Phi_2$
      \begin{enumerate}
  \item[a)] $\|(T^{\ast})^{n_{k}}\varphi\|\|M_{n_{k}}\phi\|\longrightarrow 0$, as $k\rightarrow +\infty$;
  \item[b)] $(T^{\ast})^{n_{k}}M_{n_{k}}\phi\longrightarrow \phi$, as $k\rightarrow +\infty$.
\end{enumerate}
Let $D$ be dense subset of $X$ and consider the sets $\Phi_{0}=span\{x\otimes \varphi / x\in D, \varphi\in\Phi_1\}$ and $\Psi_{0}=span\{y\otimes \phi / x\in D, \phi\in\Phi_2\}$ and the maps $N_{n_{k}} : \Psi_{0}\longrightarrow J$ define by $$N_{n_{k}}(\sum_{j=1}^{N}\beta_{j}y_{j}\otimes \phi_{j})=\sum_{j=1}^{N}\beta_{j}y_{j}\otimes M_{n_{k}}\phi_{j}.$$
By  Lemma \ref{lemma1},  $\Phi_{0}$ and $\Psi_{0}$ are subsets of $J$ which are $\|.\|_{J}$-dense  in $J$. Let  $\displaystyle A=\sum_{i=1}^{N_{1}}\alpha_{j}x_{i}\otimes \varphi_{i}\in \Phi_{0}$ and $\displaystyle  B=\sum_{j=1}^{N_{2}}\beta_{j}y_{j}\otimes \phi_{j}\in \Psi_{0}$, then
\begin{eqnarray*}
  \|(R_{T})^{n_k}A\|_{J}\|N_{n_{k}}B\|_{J} &=& \|(R_{T})^{n_k}(\sum_{i=1}^{N_{1}}\alpha_{i}x_{i}\otimes \varphi_{i})\|_{J}\|N_{n_{k}}(\sum_{j=1}^{N_2}\beta_{j}y_{j}\otimes \phi_{j})\|_{J} \\
  &=& \|\sum_{i=1}^{N_{1}}\alpha_{i}x_{i}\otimes (T^{\ast})^{n_k}\varphi_{i}\|_{J}\|\sum_{j=1}^{N_2}\beta_{j}y_{j}\otimes M_{n_{k}}\phi_{j})\|_{J} \\
   &\leq& (\sum_{i=1}^{N_{1}}|\alpha_{i}|  \|x_{i}\otimes (T^{\ast})^{n_k}\varphi_{i}\|_{J})(\sum_{j=1}^{N_2}|\beta_{j}| \|y_{j}\otimes M_{n_{k}}\phi_{j}\|_{J}) \\
  &=&  (\sum_{i=1}^{N_{1}}|\alpha_{i}|  \|x_{i}\| \|(T^{\ast})^{n_k}\varphi_{i}\|)(\sum_{j=1}^{N_2}|\beta_{j}| \|y_{j}\| \|M_{n_{k}}\phi_{j}\|)\\
    &=&\sum_{i\leq N_{1}; j\leq N_{2}}|\alpha_{j}| |\beta_{j}|  \|(T^{\ast})^{n_k}\varphi_{i}\|  \|M_{n_{k}}\phi_{j}\| \|x_{i}\|\|y_{j}\|.
\end{eqnarray*}
Using the assumption $a)$. We show that $\|(R_{T})^{n_k}A\|_{J}\|N_{n_{k}}B\|_{J}\longrightarrow0$,  as $k\longrightarrow +\infty$. In the other hand we have
\begin{eqnarray*}
  (R_{T})^{n_k}N_{n_{k}}(B) &=& (R_{T})^{n_k}N_{n_{k}}(\sum_{j=1}^{N_{2}}\beta_{j}y_{j}\otimes \phi_{j}) \\
   &=&  (R_{T})^{n_k}(\sum_{j=1}^{N_{2}}\beta_{j}y_{j}\otimes M_{n_k}\phi_{j})\\
  &=&\sum_{j=1}^{N_{2}}\beta_{j}y_{j}\otimes (T^{\ast})^{n_k}M_{n_k}\phi_{j} \\
   &\longrightarrow& \sum_{j=1}^{N_{2}}\beta_{j}y_{j}\otimes \phi_{j}=B  \ \ as \ k\longrightarrow +\infty,
\end{eqnarray*}
by using the assumption $b)$. Hence $R_{T}$ satisfies the supercyclicity criterion on $(J, \|.\|_{J})$. Thus $R_{T}$ is supercyclic on $(J, \|.\|_{J})$.\\
$\Leftarrow)$ Suppose that $R_{T}$ is supercyclic on on $(J, \|.\|_{J})$. Let  $x_{1}^{\ast}, x_{2}^{\ast}\in X^{\ast}$ are linearly independent  and define $$\begin{array}{ccccc}
\phi & : &  J & \longrightarrow & X^{\ast}\bigoplus X^{\ast}\\
 & & R & \longmapsto & R^{\ast}x^{\ast}_{1} \oplus R^{\ast}x^{\ast}_{2}.  \\
\end{array}$$
Then $\phi$ is surjective, indeed, let $y_{1}^{\ast}, y_{2}^{\ast}\in X^{\ast}$, we take $x_{1}^{\ast}, x_{2}^{\ast}\in X^{\ast}$ such that $x_{i}^{\ast}(x_{j})=\delta_{i, j}$ and put $R=x_{1}\otimes y_{1}^{\ast}+x_{2}\otimes y_{2}^{\ast}$, then $\phi(R)=(R^{\ast}x_{1}^{\ast}, R^{\ast}x_{2}^{\ast})=(x_{1}^{\ast}\circ R, x_{2}^{\ast}\circ R)=(y_{1}^{\ast}, y_{2}^{\ast})$.
For  $A\in J$, we have \begin{eqnarray*}
  (\phi\circ R_{T})A &=&  \phi(AT)  \\
  &=& (AT)^{\ast}x_{1}^{\ast}\oplus (AT)^{\ast}x_{2}^{\ast} \\
  &=&T^{\ast}A^{\ast}x_{1}^{\ast}\oplus T^{\ast}A^{\ast}x_{2}^{\ast} \\
  &=& (T^{\ast}\oplus T^{\ast})(A^{\ast}x_{1}^{\ast}\oplus A^{\ast}x_{2}^{\ast})) \\
  &=&  ((T^{\ast}\oplus T^{\ast})\circ\phi)A
  \end{eqnarray*}

 Therefore, $\phi\circ R_{T}=(T^{\ast}\oplus T^{\ast})\circ\phi$. Thus $T^{\ast}\oplus T^{\ast}$ is supercyclic on $X^{\ast}\bigoplus X^{\ast}$. Hence, by \cite[Lemma 3.1]{T. Bermudez} $T^{\ast}$ satisfies the supercyclicity criterion on $X^{\ast}$.
\end{proof}
We have the following corollary.
\begin{corollary}
Let $X$ be  Banach space, with $\dim X>1$ such that $X^{\ast}$, its dual, is separable and $T\in \mathcal{B}(X)$. Then the following are equivalent :
 \begin{enumerate}
   \item[$(i)$]$T^{\ast}$ satisfies the supercyclicity criterion on $X^{\ast}$.
   \item[$(ii)$] $R_{T}$ is supercyclic on $(\mathcal{K}(X),\|.\|)$  endowed with the norm operator topology.
   \item[$(iii)$]$R_{T}$ is supercyclic on $\mathcal{B}(X)$ in the strong operator topology.
    \end{enumerate}
\end{corollary}
\begin{proof}
$(i)\Leftrightarrow(ii)$ Consequence of  Theorem \ref{theorem2}. Since $\mathcal{K}(X)$ is an admissible Banach ideal of $\mathcal{B}(X)$.\\
 $(i)\Rightarrow(iii)$ Suppose that $T^{\ast}$ satisfies the supercyclicity criterion on $X^{\ast}$. Let $U$ and $V$ be two non-empty open subsets of $\mathcal{B}(X)$ in the strong operator topology. Since $\mathcal{K}(X)$ is dense in $\mathcal{B}(X)$ with the strong operator topology \cite[Corollary 3]{Chan2001}, there exist $A_{1}, A_{2}\in\mathcal{K}(X)$ such that $A_{1}\in U$ and $A_{2}\in V$. Thus we can find $x_1, x_2\in X\backslash\{0\}$ and $\epsilon_1, \epsilon_2>0$ such that $$\{A\in\mathcal{B}(X) \ : \ ||(A-A_1)x_1||<\epsilon_1\}\subset U$$
and $$\{A\in\mathcal{B}(X) \ : \ ||(A-A_2)x_2||<\epsilon_2\}\subset U.$$
Let $$U_i =\{A\in\mathcal{K}(X) \ : \ ||A-A_i||<\frac{\epsilon_i}{||x_i||}\}.$$
$U_i$ is a non-empty open subset of $\mathcal{K}(X)$ with the norm operator topology. By Theorem \ref{theorem2} with $J=\mathcal{K}(X)$, $R_{T}$ is supercyclic on $(\mathcal{K}(X), ||.||)$, so there is some $\alpha\in\mathbb{C}$ and $n\in \mathbb{N}$ such that $$\alpha(R_{T})^{n}U_{1}\cap U_{2}\neq\emptyset.$$
Hence, it follows that $\alpha(R_{T})^{n}U\cap V\neq\emptyset$. Thus $R_{T}$ is supercyclic on $\mathcal{B}(X)$ in the strong operator topology.\\
$(i)\Rightarrow(iii)$ By the same technique  as in the proof of  Theorem \ref{theorem2}.
\end{proof}
\section{Stability of supercyclicity tensor product.}
In \cite{A. Peris} the authors gave a sufficient condition for the tensor product $T\widehat{\otimes}R$ of two operators to be  hypercyclic. We inspired from this results, we give a sufficient condition to the tensor product $T\widehat{\otimes}R$ of two operators to be  supercyclic.
\begin{definition}
Let $X$ be Banach space.  An operator $T\in\mathcal{B}(X)$ is said to be  satisfies the Tensor Sypercyclicity Criterion (TSC) if there exists dense subsets $D_1, D_2\subset X $,  an increasing sequence  $(n_{k})_{k}$  of positive integers, $(\lambda_{n_{k}})_{k\in\mathbb{N}}\subset\mathbb{C}\backslash\{0\}$ and a sequence of mappings  $S_{n_{k}}$ : $D_2\longrightarrow X$ such that :
\begin{enumerate}
  \item[$(i)$] $(\lambda_{n_{k}}T^{n_{k}}x)_{k\in\mathbb{N}}$ is bounded  for all $x\in D_1$;
  \item[$(ii)$] $(\frac{1}{\lambda_{n_{k}}}S_{n_{k}}y)_{k\in\mathbb{N}}$ is bounded for all $y\in D_2$;
  \item[$(iii)$] $(T^{n_{k}} S_{n_{k}}y)_{k\in\mathbb{N}}\longrightarrow y$ for all $y\in D_2$.
\end{enumerate}
\end{definition}
\begin{example}\label{example1}
\begin{enumerate}
  \item Clearly, a sequences of operators satisfying the  Sypercyclicity Criterion satisfy Tensor Sypercyclicity Criterion .
  \item The identity map on $X$ satisfy the Tensor Sypercyclicity Criterion .
  \item Any isometry on a Banach space satisfies the Tensor Sypercyclicity Criterion  with respect to the sequence of all positive integers.
\end{enumerate}
\end{example}
\begin{theorem}\label{theorem3}
Let  $E$ and $F$ be separable  Banach spaces. If $T_{1}\in\mathcal{B}(E)$   satisfies the  Sypercyclicity Criterion and $T_{2}\in\mathcal{B}(F)$   satisfies the Tensor  Sypercyclicity Criterion, then $$T_{1}\widehat{\otimes}_{\pi}T_{2} \ : \ E\widehat{\otimes}_{\pi}F\longrightarrow E\widehat{\otimes}_{\pi}F$$
satisfies the Sypercyclicity Criterion. Accordingly, it is supercyclic.
\end{theorem}
\begin{proof}
Let $X_{1}, X_{1}\subset E$, $Y_{1}, Y_{2}\subset F$ be dense subspaces, $(\lambda_{n_{k}}^{1})_{k\in \mathbb{N}}, (\lambda_{n_{k}}^{2})_{k\in \mathbb{N}}\subset\mathbb{C}\backslash\{0\}$ and $S_{n_{k}}^{1} : X^{2} \longrightarrow E$, $S_{n_{k}}^{2}  : Y^{2} \longrightarrow F$, $k\in \mathbb{N}$, linear  maps satisfying the conditions of   Sypercyclicity Criterion and Tensor  Sypercyclicity Criterion for $T_{1}\in\mathcal{B}(E)$ and $T_{2}\in\mathcal{B}(F)$, respectively. We will see that $Z_1:=X_{1}\otimes Y_{1}$,  $Z_2:=X_{2}\otimes Y_{2}$, $(\lambda_{n_{k}}:=\lambda_{n_{k}}^{1}.\lambda_{n_{k}}^{2})_{k\in \mathbb{N}}$ and the maps $S_{n_{k}}:=S_{n_{k}}^{1}\otimes S_{n_{k}}^{2} \ : \ Z_{2} \longrightarrow E\otimes F, $ are such that conditions  of the Sypercyclicity Criterion are satisfied for the operator $$T:=T_{1}\widehat{\otimes}T_{2} \ : E\widehat{\otimes}_{\pi} F\longrightarrow E\widehat{\otimes}_{\pi} F .$$
Indeed, if we compute on elementary tensors, we obtain for every $x_1\in X_1$, $x_2\in Y_1$, $y_{1}\in X_2$ and $y_2\in Y_2$ :
\begin{eqnarray*}
 \displaystyle\lim_{k\rightarrow +\infty}\Pi(\lambda_{n_{k}}T^{n_{k}})(x_{1}\otimes x_{2})) &=& \displaystyle\lim_{k\rightarrow +\infty}\Pi(\lambda_{n_{k}}^{1}.\lambda_{n_{k}}^{2}(T^{n_{k}}_{1}\otimes T^{n_{k}}_{2})(x_{1}\otimes x_{2}))\\
 &=& \displaystyle\lim_{k\rightarrow +\infty}\Pi((\lambda_{n_{k}}^{1}T^{n_{k}}_{1}\otimes \lambda_{n_{k}}^{2}T^{n_{k}}_{2})(x_{1}\otimes x_{2}))\\
  &=&\displaystyle\lim_{k\rightarrow +\infty}\|\lambda_{n_{k}}^{1}T^{n_{k}}_{1}x_{1}\| \|\lambda_{n_{k}}^{2}T^{n_{k}}_{2}x_{2}\|   \\
  &=& 0,
  \end{eqnarray*}
since the first sequence tends to 0 and the second one is bounded. Analogously
\begin{eqnarray*}
 \displaystyle\lim_{k\rightarrow +\infty}\Pi(\frac{1}{\lambda_{n_{k}}}S_{n_{k}})(y_{1}\otimes y_{2})) &=& \displaystyle\lim_{k\rightarrow +\infty}\Pi(\frac{1}{\lambda_{n_{k}}^{1}.\lambda_{n_{k}}^{2}}(S_{n_{k}}^{1}\otimes S_{n_{k}}^{2})(y_{1}\otimes y_{2}))\\
 &=& \displaystyle\lim_{k\rightarrow +\infty}\Pi((\frac{1}{\lambda_{n_{k}}^{1}}S_{n_{k}}^{1}\otimes \frac{1}{\lambda_{k_{k}}^{2}}S_{n_{k}}^{2})(y_{1}\otimes y_{2}))\\
  &=&\displaystyle\lim_{k\rightarrow +\infty}\|\frac{1}{\lambda_{n_{k}}^{1}}S_{n_{k}}^{1}y_{1}\| \|\frac{1}{\lambda_{n_{k}}^{2}}S_{n_{k}}^{2}y_{2}\| \\
  &=& 0,
\end{eqnarray*}
since the first sequence tends to 0 and the second one is bounded. Finally,
\begin{eqnarray*}
  \displaystyle\lim_{k\rightarrow +\infty}\Pi[(T^{n_{k}}S_{n_{k}})(y_{1}\otimes y_{2})-(y_{1}\otimes y_{2})]  &=& \displaystyle\lim_{k\rightarrow +\infty}\Pi[(T^{n_{k}}_{1}\otimes T^{n_{k}}_{2})(S_{n_{k}}^{1}\otimes S_{n_{k}}^{2})(y_{1}\otimes y_{2})-(y_{1}\otimes T^{n_{k}}_{2}S_{n_{k}}^{2}y_{2})+\\
  &&(y_{1}\otimes T^{n_{k}}_{2}S_{n_{k}}^{2}y_{2})-(y_{1}\otimes y_{2})]  \\
  &=& \displaystyle\lim_{k\rightarrow +\infty}\Pi[((T^{n_{k}}_{1}S_{n_{k}}^{1}y_{1}-y_{1})\otimes T^{n_{k}}_{2}S_{n_{k}}^{2}y_{2})+y_{1}\otimes (T^{n_{k}}_{2}S_{n_{k}}^{2}y_{2}-y_{2})]\\
  &\leq& \displaystyle\lim_{k\rightarrow +\infty}\{\Pi[(T^{n_{k}}_{1}S_{n_{k}}^{1}y_{1}-y_{1})\otimes T^{n_{k}}_{2}S_{n_{k}}^{2}y_{2}]+\Pi[y_{1}\otimes (T^{n_{k}}_{2}S_{n_{k}}^{2}y_{2}-y_{2})]\}\\
   &=& \displaystyle\lim_{k\rightarrow +\infty}\{\|T^{n_{k}}_{1}S_{n_{k}}^{1}y_{1}-y_{1}\| \|T^{n_{k}}_{2}S_{n_{k}}^{2}(y_{2})\|+\|y_{1}\| \|T^{n_{k}}_{2}S_{n_{k}}^{2}y_{2}-y_{2}\|\} \\
   &=& 0,
\end{eqnarray*}
which completes the proof by taking linear combinations of elementary tensors.
\end{proof}
In the following  Proposition, we show  the connection between supercyclicity of tensor products and supercyclicity of direct sums, and yields another equivalent formulation in the context of tensor products of the supercyclicity criterion.
\begin{proposition}
Let $E, F$ be a separable Banach spaces with $\dim F\geq2$ and  $T\in\mathcal{B}(E)$. The following are equivalent :
\begin{enumerate}
  \item [$(i)$] $T$ satisfies the supercyclicity criterion.
  \item [$(ii)$]$T\widehat{\otimes}I : E\widehat{\otimes}_{\pi}F\longrightarrow E\widehat{\otimes}_{\pi}F $ is supercyclic for  the projective tensor norm $\pi$.
  \item [$(iii)$]$T\oplus T: E\bigoplus E\longrightarrow E\bigoplus E$ is supercyclic.
\end{enumerate}
\end{proposition}
\begin{proof}
$(i)\Rightarrow(ii)$ Is a consequence of Theorem \ref{theorem3} by taking $T_2=I$.\\
$(iii)\Rightarrow(i)$ See \cite[Lemma 3.1]{T. Bermudez}.\\
$(ii)\Rightarrow(iii)$ Since $\dim F\geq2$, let $x_{1}^{\ast}, x_{2}^{\ast}\in F^{\ast}$ and consider the following commutative diagram :
 $$  \xymatrix{
    E\widehat{\bigotimes_{\pi}}F \ar[r]^{T\widehat{\otimes}I} \ar[d]_\varphi  & E\widehat{\bigotimes_{\pi}}F \ar[d]^\varphi \\
    E\bigoplus E \ar[r]_{T\oplus T} & E\bigoplus E
  }$$
where $\displaystyle\varphi(\sum_{i\leq N} e_{i}\otimes x_{i}):=(\sum_{i\leq N} \langle x_{i}, x_{1}^{\ast}\rangle e_{i}, \sum_{i\leq N} \langle x_{i}, x_{2}^{\ast}\rangle e_{i})$. $\varphi$ is surjective. Indeed, Let $e_{1}, e_{2}\in E$, we take $x_{1}, x_{2}\in F$ such that $x_{i}^{\ast}(x_{j})=\delta_{i, j}$, so we have $\varphi(e_{1}\otimes x_{1}+e_{2}\otimes x_{2})=(e_{1}, e_{2})$.
Let $\displaystyle u=\sum_{i\leq N} e_{i}\otimes x_{i}\in E\widehat{\bigotimes_{\pi}}F$, then
\begin{eqnarray*}
  (T\oplus T\circ \varphi)(u) &=& (T\oplus T\circ \varphi)(\sum_{i\leq N} e_{i}\otimes x_{i}) \\
   &=&  (T\oplus T)(\sum_{i\leq N} \langle x_{i}, x_{1}^{\ast}\rangle e_{i}, \sum_{i\leq N} \langle x_{i}, x_{2}^{\ast}\rangle e_{i}) \\
   &=& (\sum_{i\leq N} \langle x_{i}, x_{1}^{\ast}\rangle Te_{i}, \sum_{i\leq N} \langle x_{i}, x_{2}^{\ast}\rangle Te_{i}) \\
    &=& (\sum_{i\leq N} (Te_{i}\otimes x_{1}^{\ast}) x_{i} , \sum_{i\leq N} (Te_{i}\otimes x_{2}^{\ast})x_{i}) \\
   &=& (\sum_{i\leq N} (T\widehat{\otimes} I)(e_{i}\otimes x_{1}^{\ast})x_{i},\sum_{i\leq N} (T\widehat{\otimes} I)(e_{i}\otimes x_{2}^{\ast})x_{i})\\
    &=& \varphi\circ(T\widehat{\otimes} I)(\sum_{i\leq N} e_{i}\otimes x_{i})\\
    &=& \varphi\circ(T\widehat{\otimes} I)(u).
\end{eqnarray*}
Thus, $T\oplus T$ is supercyclic on $E\bigoplus E$.
\end{proof}

\end{document}